\documentclass[12pt]{amsart}

\usepackage{amscd,amssymb,amsfonts,amsmath,amstext,amsthm}
\usepackage{setspace}
\usepackage{cleveref}
\usepackage[usenames]{color} 

\newtheorem{thm}{Theorem}[section]
\newtheorem{df}[thm]{Definition}

\newtheorem{ex}[thm]{Example}
\newtheorem{cor}[thm]{Corollary}
\newtheorem{prop}[thm]{Proposition}
\newtheorem{lem}[thm]{Lemma}

\newtheorem{conj}[thm]{Conjecture}
\numberwithin{equation}{section}

\newcommand{\num}{\refstepcounter{thm}}

\newcommand{\ran}{\rangle}
\newcommand{\lan}{\langle}

\newcommand{\BC}{\mathbb{C}}

\renewcommand\k{\Bbbk}

\renewcommand{\a}{\alpha}

\newcommand\inv{^{-1}}

\newcommand{\ltri}{\triangleleft}
\newcommand{\rtri}{\triangleright}
\newcommand{\dimn}{\operatorname{dim}}
\newcommand{\lharp}{\leftharpoonup}
\newcommand{\rharp}{\rightharpoonup}

\newcommand\WTS[1]{\widetilde{S}_{#1} }
\newcommand{\fix}{\operatorname{fix}}

\begin{document}

\title[Indicators of Bismash Products]{Indicators of Bismash Products from Exact Symmetric Group Factorizations}

\author{Joseph Timmer}\address{Department of Mathematics
\\ Louisiana State University
\\ Baton Rouge, LA 70803-4918}\email{jtimme1@lsu.edu}

\begin{abstract}
We prove that all irreducible representations of the bismash product $H = \k ^G \# \k S_{n-k}$ have Frobenius-Schur indicators +1 or 0 where $\k$ is an algebraically closed field and $S_n = S_{n-k}\cdot G$ is an exact factorization.  Moreover, we have an description of those simple modules which are not self-dual.  We prove some new results for bismash products in general and make conjectures about bismash products that arise from other exact factorizations of $S_n$.
\end{abstract}

\maketitle
\setcounter{section}{0}
\section*{Introduction}
In this work, we study the representations of bismash products that are related to exact factorizations of the symmetric group $S_n$ over an algebraically closed  field $\k$. We are mainly concerned with the Frobenius-Schur indicators of such Hopf algebras. The classical Frobenius-Schur theorem for  nite groups was extend to semisimple Hopf algebras in \cite{LM}. Since then, the indicator has proven itself useful for the general theory of Hopf algebras. It is used in classifying Hopf algebras of small dimension, in  finding the dimensions of simple modules \cite{KSZ1} and determining the prime divisors of the exponent and establishing a ``Cauchy's Theorem" for semisimple Hopf Algebras \cite{KSZ2}.

In 2009, Jedwab and Montgomery \cite{JM}, proved the Hopf algebra $H = \k^{C_n}\#\k S_{n-1}$ arising from the exact factorization $S_n = S_{n-1}\cdot C_n$, the Schur indicator value of any irreducible H-module is 1, that is, H is \emph{totally orthogonal}. This result is known classically for $S_n$, and was also shown to be true for the Drinfeld double $D(S_n)$ in \cite{KMM}. The natural question to ask is if the same is true for other exact factorizations of $S_n$.

The main result is \Cref{TGfact} which extends the result of \cite{JM} to bismash product Hopf algebras $\k^G \# \k S_{n-k}$ arising from factorizations $S_n = S_{n-k}\cdot  G$ where $S_{n-k}$ fixes $k$ many distinct points and G is sharply $k$-transitive. The result in \cite{JM} is a special case of $k = 1$.  We also greatly reduce the work towards a more general conjecture concerning Schur indicators for such factorizations of $S_n=FG$ with a result about bismash products in general, \Cref{CFHIT}. We give an interesting example, \Cref{cntrex} about bismash products that shows some of the results in this paper cannot be further extended.  We conclude with the aforementioned conjecture, \Cref{conjecture} and some ideas of its proof.


\section{Notation}
 Throughout this work, we work over an algebraically closed field $\k$ of characteristic $0$.  All groups and sets mentioned are assumed to be finite. 
 
For a group $G$ acting on a set $X$, we write $x^g$ for $x\in X$ and $g\in G$ to be the image of $x$ under the action of $g$.  This notation is to be read left to right, thus giving $x^{gh} = (x^g)^h$ for $g,h\in G$.  For a subset $Y\subset X$, we write $Y^G= \{y^g : y\in Y, g\in G\}$.  For a single element $x$, we use the notation $\mathcal{O}_x$ to denote the orbit of $x$.  

We let $\Omega=\{1,2,\ldots, n\}$ and denote the symmetric group on $\Omega$ by $S_n$ or $S_\Omega$.  We use a similar notation to denote the alternating group $A_n$ or $A_\Omega$.  We refer to elements of $A_n$ as even permutations and odd permutations for the other elements of $S_n$.  For a subset $\Delta\subset\Omega$, we use a similar notation to define $S_\Delta$ and $A_\Delta$. For an abstract finite set $X$, we use $Sym(X)$ to denote the group of bijections from $X$ to itself.

Our notation for elements in $S_n$ and how we compute products in $S_n$ will be that used in \cite{GAP}.  Entries in a cycle are separated by commas, e.g., $(1,2,3,4)$, and products will also to be read left to right, e.g., $(1,2)(1,3) = (1,2,3)$.  We use $( )$ to denote the identity element of $S_n$.  Conjugation of an element $\sigma \in S_n$ by $x\in S_n$ will be denoted by $\sigma^x = x\inv \sigma x$.  An \textbf{involution} is a permutation $\sigma\in S_n$ such that $\sigma^2 = ()$

For any set $X$, we denote, for $x,y\in X$, $\delta_{x,y} = \begin{cases} 1 & x=y \\ 0 & x\neq y \end{cases}$ the \textbf{Kronecker delta symbol}.

\section{Exact Factorizations of $S_n$}
In this section, we recall some elementary results in group theory; specifically those pertaining the the symmetric group.  We state the relevant definitions and establish notations that will be used in some important results that broaden the scope of our desired results about bismash products $\k^G\# \k F$.  We include the proofs of some known results, as a suitable reference was not found for everything.
\begin{df} \rm{
An \textbf{exact factorization of a group} $L$ is a pair of proper subgroups $F,G$ such that $L=FG$ and $F\cap G=\{e\}$.  We then say $F$ and $G$ \textbf{factor $L$ exactly}.
} \end{df}

\begin{lem}\label{FGequiv} A pair of proper subgroups $F,G \subset L$ factor $L$ exactly if and only if $|F|\cdot |G| = |L|$ and $F\cap G= \{ e \}$. \end{lem}
\begin{proof}
Immediate from $|FG|=\frac{|F|\cdot |G|}{|F\cap G|}$. \end{proof}

\begin{cor} If $L=FG$ is an exact factorization of $L$, then so is $L=GF$. \end{cor}

\begin{cor}\label{EFequiv} $L=FG$ is an exact factorization if and only if for each $l\in L$, $l=fg$ for some uniquely determined $f\in F$, $g\in G$.
\end{cor}

It is clear if one has an automorphism $\phi$ of $L$, then $L=\phi(F)\phi(G)$ is also an exact factorization.  This next lemma establishes if one has an exact factorization $L=FG$, one can also form new factorizations by conjugation of the exact factors.
\begin{lem}\label{EFconj} Let $L=FG$ be an exact factorization of $L$ and $x\in L$.  Let $F^x = \{ x\inv f x : f\in F \}$.  Then $L=F^x G$ is an exact factorization of $L$. \end{lem}

\begin{proof}
Let $z\in F^x \cap G$.  Then $z= x\inv f x = (f_1g_1)\inv f (f_1 g_1)$ for some $f,f_1\in F$, $g_1\in g$.  Thus, $g_1 z g_1\inv  = f_1\inv f f_1$.  As $z\in G$, we have $f_1\inv f f_1 \in F\cap G$ and so $f=e$.  Thus, $z=e$. \end{proof}

\begin{cor} If $L=FG$ is an exact factorization of groups, then so is $L=F^x G^y$ for any $x,y\in L$. \end{cor}

With an exact factorization of $L$, we can define set maps $\rtri : G\times F \to F$ and $\ltri : G\times F \to G$ as follows.  Since $L=FG$, every element $l\in L$ can be written $l=f'g'$ for some $f'\in F$, $g'\in G$.  Similarly, every $l\in L$ can be also uniquely expressed as $l=gf$ for $g\in G$ and $f\in F$.  Declare $g\rtri f:= f'$ and $g\ltri f := g'$ for $f'g'=l=gf$.  By uniqueness of the decompositions, these set maps are well defined.  Note that we have $e\ltri f = e$ and $g\rtri e = e$ for all $f\in F$, $g\in G$.

\begin{prop}\label{MPA} For an exact factorization $L=FG$, the set maps $\rtri : G\times F \to F$ and $\ltri : G\times F \to G$ given by 
$g\rtri f:= f'$ and $g\ltri f := g'$ for $f'g'=l=gf$ are group actions.   Moreover, these actions satisfy, for $a,b\in F$, $x,y\in G$,
\begin{enumerate}
\item $x\rtri (ab) = (x\rtri a)((x\ltri a)\rtri b)$,
\item $(xy)\ltri a = (x\ltri (y\rtri a) )(y\ltri a).$
\end{enumerate}\end{prop}

\begin{proof} 
\cite{T}.
\end{proof}

\begin{df} \rm{ Let $F,G$ be groups with group actions $\rtri : G\times F \to F$ and $\ltri : G\times F \to G$ satisfying the conditions of \Cref{MPA}.  We then say $(F,G,\rtri, \ltri)$ form a \textbf{matched pair of groups}. } \end{df}
 
 Takeuchi shows in \cite{T} that exact factorizations of a group $L=FG$ are equivalent to a matched pair of groups $(F,G,\rtri,\ltri)$.

Now we focus on exact factorizations of the symmetric group $S_n$. These descriptions are based on a theorem due to Wiegold and Williamson, \cite{WW}, which describes the structure of groups $F,G$ which may factor $S_n$ exactly.  Upon reading the paper, we found some inaccuracies in the conclusions of the main theorem.  We point them out accordingly, namely the conclusion that if $F|_\Gamma \cong S_\Gamma$, then $F|_\Delta =1$.  We restate the theorem with this in mind, and then revisit what these factor groups are.  We do not know if this is an actual oversight in \cite{WW}.  We fix this minor issue in \Cref{WWfix}, and it may be that \cite{WW} did not consider this situation to give a  ``new" factorizations.  

\begin{ex} $S_n=A_nT$ where $T$ is generated by an odd involution.  \end{ex}

\begin{ex}\label{GsubgroupSn} Let $G$ be a group with $|G|=n$.  Then $S_n=G S_{\Omega-\alpha}$ is an exact factorization of $S_n$ for any $\alpha\in\Omega$. \end{ex}
\begin{proof}
Let $G\leq S_n$ via the regular action of $G$ on itself.  Then, by identifying $G$ as a permutation group, $G$ is regular on $\Omega$.  Thus, $G'\cap S_{\Omega-\alpha} = \{ ( ) \}$ for any $\alpha \in \Omega$ and by \Cref{FGequiv}, $S_n=G' S_{\Omega-\alpha}$ is an exact factorization of $S_n$. \end{proof}

We now simply assume that $G\leq S_n$ for some $n$ and hence $G$ will act naturally on $\Omega=\{1,2,\ldots, n\}$.

\begin{df} \rm{Let $G\leq S_n$ be a permutation group.  The \textbf{degree} of $G$ is the number of points moved by $G$, i.e., $\left| \Omega^ G\right| = \left| \{ i^g : 1\leq i\leq n,\; g\in G\}\right|$.} \end{df}

Observe that transitive subgroups of $S_n$ have degree $n$.  We now recall some generalizations of transitive group actions on sets.

\begin{df} \rm{ Let $G$ be a group acting on a set $X$.  We say $G$ is \textbf{$k$-homogeneous} if for each pair of subsets $Y,Z\subset X$ with $|Y|=|Z|=k$, there is an element $g\in G$ such that $Y^g=Z$. } \end{df}

\begin{df} \rm{ Let $G$ be a group acting on a set $X$ and $k\geq 1$.  We say $G$ is \textbf{$k$-transitive} if for each pair of duplicate-free $k$-tuples $(x_1, x_2, \ldots, x_k), (y_1, y_2, \ldots, y_k) \in X^k$, there exists $g\in G$ such that $x_i^g=y_i$ for $1\leq i \leq k$.  We say $G$ is \textbf{sharply $k$- transitive} if such $g$ is unique. } \end{df}

Note that if $G$ is $k$-transitive, then $G$ is also $k-1$-transitive.  In particular, any $k$-transitive group is transitive. 

\begin{df} \rm{ Let $1\leq k \leq n$.  A subgroup $G\leq S_n$ has the \textbf{k fixed point property} if an element $g\in G$ having k fixed points implies that $g=( )$. } \end{df}

\begin{prop}\label{S-TranOrder} A subgroup $G\leq S_n$ is sharply $k$-transitive on $\Omega$ if and only if $|G|=n(n-1)\cdots(n-k+1)$ and $G$ has the $k$ fixed point property. \end{prop}
\begin{proof}
Suppose $G$ is sharply $k$-transitive.  For each $(\alpha_1, \alpha_2, \ldots, \alpha_k)$ such that $\a_i\neq\a_j$ for $i\neq j$, there is a unique $g\in G$ such that $(1,2,\ldots, k)^g=(\a_1, \a_2,\ldots,\a_k)$.  The set of such $(\a_1,\a_2,\ldots,a_k)$ has cardinality $n(n-1)\cdots(n-k+1)$ and so $G$ must have this same cardinality.  If $g\in G$ has $k$ fixed points, then $g$ sends the $k$-tuple of its fixed points to itself, hence $g=( )$.

Suppose $|G|=n(n-1)\cdots(n-k+1)$ and $G$ has the $k$ fixed point property.  It suffices to show the action of $G$ on ordered $k$-tuples of $\Omega^k$ is free and transitive.  But the $k$-fixed point property is exactly that the action is free and thus since $|G|=n(n-1)\cdots(n-k+1)$, the action is also transitive. \end{proof}

\begin{ex} The group $S_n$ is sharply $n$ and $n-1$-transitive and the group $A_n$ is sharply $n-2$ transitive for $n\geq3$ \end{ex}

We now state a classical result about sharply 4-transitive and 5-transitive groups.

\begin{thm}\cite[1873]{Jo}\label{Jo} The only sharply 4-transitive groups are $S_4$, $S_5$, $A_6$ and $M_{11}$, the Mathieu 11-group.  The only sharply 5-transitive groups are $S_5$, $S_6$, $A_7$ and $M_{12}$, the Mathieu 12-group. \end{thm}

\begin{df} \rm{ Let $\Delta\subseteq \Omega$ and $G$ a permutation group.  We say $\Delta$ is a \textbf{stable block of $G$} if $\Delta^G=\Delta$. } \end{df}
 Note that we do not require that an element of $G$ fixes each individual element of a stable block $\Delta$, merely that each $g\in G$ restricts to a bijection on $\Delta$.
 
 \begin{df} \rm{ For $\Delta\subseteq\Omega$ and a permutation group $G$, we denote \textbf{the stabilizer of} $\Delta$ by $G_\Delta = \{g\in G : \alpha^g=\alpha\;\forall\; \alpha\in\Delta\}$.  } \end{df}
 
\begin{lem} For $\Delta\subseteq\Omega$ a stable block of $G$, $G_\Delta$ is a normal subgroup of $G$. \end{lem}

\begin{df} \rm{ For $\Delta\subseteq\Omega$ a stable block of a permutation group $G\subset S_n$, let $G|_\Delta$ be the restriction of $G$ to the subset $\Delta$. (In \cite{W}, this group is denoted by $G^\Delta$.) } \end{df}

\begin{lem} For $\Delta\subset\Omega$ a stable block of a permutation group $G$, we have a group isomorphism $G|_\Delta \cong G/ G_{\Delta}$. \end{lem}
\begin{proof}
The map $G\to G|_\Delta$ given by restriction of functions is clearly a surjective group homomorphism with kernel $G_\Delta$. \end{proof}

Let $K$ be a finite field of size $q$ and $d$ a natural number.  We use the notation of $GL(n,q)$ to denote the general linear group of rank $n$ over $K$.  We use a similar notation $SL(n,q)$, $PGL(n,q)$ and $PSL(n,q)$ for the special linear, projective general linear and  projective special linear groups respectively.

\begin{df}\label{affgrps} \rm{ Let $K$ be a finite field of size $q$ and $\mathcal{G}$ to be the Galois group of $K$ over its prime subfield. Define:
\begin{enumerate}
\item $AGL(1,q)$, is the group of all functions $K \to K$ given by $x\mapsto gx+y$ for $g\in K^\times $ and $y\in K$.
\item $ASL(1,q)$, is the subgroup of even permutations in $AGL(1,q)$.
\item $A\Gamma L(1,q)$, is the group of all functions $K\to K$ given by $x\mapsto gx^\sigma +y$ for some $g\in K^\times$, $y\in K$ and $\sigma \in \mathcal{G}$.
\end{enumerate} }
\end{df}

\begin{df} \rm{ Let $V$ be a dimension $d$ vector space over a finite field $K$ of size $q$.  Let $\mathcal{G}$ be the Galois group of $K$ over it's prime subfield.  Let $A\Gamma L(d,q)$ be the group of \textbf{semilinear transformations}, that is, a bijection $f: V\to V$ such that for all $x,y\in V$, $\lambda \in K$ $$f(x + y) = f(x) + f(y)$$ and $$f(\lambda x) = \lambda^\sigma f(x)$$ for some $\sigma \in \mathcal{G}$. } \end{df}

\begin{df} \rm{ Let $V$ be a dimension $d$ vector space over a finite field $K$ of size $q$.  Denote by $Z$ the group of all nonzero scalar transformations of $V$.  Then define the quotient group $P\Gamma L (d,q) = A\Gamma L (d,q) / Z$. }
\end{df}

For the remainder of this section, $p$ will denote a specific prime, which will be based on $n=| \Omega |$.  For $n=4,5$, $p=3$; for $n=6,7,8$, $p=5$ and for $n\geq 8$, $p$ is the largest prime such that $p < n-2$.  The classical Bertrand-Tchebychev theorem concludes that $p \geq n/2$.  Now let $S_n=FG$ be an exact factorization of $S_n$.  Since $p | n!$, $S_n$ must have an element of order $p$ and since $p \geq n/2$, this element must be a $p$-cycle.  Moreover, $p$ must also divide $|F|$ or $|G|$ whereby $F$ or $G$ contains a $p$-cycle.  We will now assume $F$ contains a $p$-cycle.  
For this section, we reserve the Greek letter $\Gamma$ to be a specific subset of $\Omega$. 

Denote by $\Gamma$ the $F$-orbit of $P:=\{1,2,\ldots, p\}$, specifically $\Gamma = P^F = \{ i^x : i\in P, x\in F\}$.  Let $\Delta = \Omega - \Gamma$ and $k = |\Delta|$.  By means of \Cref{EFconj}, we may assume that $\Gamma = \{1,2,\ldots n-k\}$. We now state the corrected main result of \cite{WW}.

\begin{thm}\label{WW}\cite[Theorem S]{WW} If $S_n=FG$ is an exact factorization, then $F$ and $G$ must satisfy one of the following conditions: . 
\begin{enumerate} 
\item $F|_\Gamma\cong A_\Gamma$.
	\begin{enumerate} 
	\item $k=0$, $G$ is generated by an odd permutation of order $2$, $F=A_\Omega$.
	\item $1\leq k \leq 3$, $G$ contains $G\cap A_\Omega$ as a sharply k-transitive subgroup of index $2$, $F|_\Delta=1$.
	\item $2\leq k \leq 5$, $G$ is sharply k-transitive, $F|_\Delta$ is generated by a transposition. 
	\end{enumerate}

\item $F|_\Gamma\cong S_\Gamma$.
	\begin{enumerate}
	\item $1\leq k \leq 5$, $G$ is sharply k-transitive.
	\item $2\leq k \leq 4$, $G$ is k-homogeneous but not $k$-transitive.  This gives rise to the following table, where the number $q$ in the last two items is a prime-power congruent to 3 modulo 4.
	\end{enumerate}
	\begin{tabular}{l l l l l l}
	\hline
	No.	&   $k$    	&      $n$		&         $G$       	&  Generators for $F$     		   \\ \hline
	1	& 	$4$	    & 	$9$	    	& $PSL(2,8)$        	& $(1,2,3,4,5)(6,7,8), (1,2), (6,7)$ 	    \\
	2	& 	$4$	   	& 	$9$	    	& $P\Gamma L(2,8)$  	& $(1,2,3,4,5)(6,7), (1,2) $		     \\
	3	& 	$4$		& 	$33$		& $P\Gamma L(2,32)$	    & $(1,2,\ldots,29)(30,31), (1,2)(30,31,32)$ \\
	4	& 	$3$		& 	$8$	    	& $AGL(1,8)$			& $(1,2,3,4,5)(6,7,8), (1,2), (6,7) $			\\ 
	5	& 	$3$		& 	$8$	    	& $A\Gamma L(1,8)$		& $(1,2,3,4,5)(6,7), (1,2)	$			\\
	6	& 	$3$		& 	$32$		& $A\Gamma L(1,32)$		& $(1,2,\ldots,29)(30,31), (1,2)(30, 31, 32)$ \\
	7	& 	$3$		& 	$q+1$   	& $PSL(2,q)$			& $(1,2,\ldots, q-2)(q-1, q), (1,2)$ \\
	8	& 	$2$		& 	$q$		    & $ASL(1,q)$			& $(1,2,\ldots, q-2)(q-1, q), (1,2)$ \\
	\hline
	\end{tabular}

\item $F|_\Gamma \ncong A_\Gamma, S_\Gamma$.
	\begin{enumerate}
	\item $k=3$, $n=8$, $G=AGL(3,2)$ of order 1344, $F|_\Gamma$ has order 5 and $F|_\Delta=S_\Delta$.
	\item $k=3$, $n=8$, $G=AGL(3,2)$, $F=\lan(1,2,3,4,5)(6,7,8), (2,5)(3,4)(6,7)\ran$ where $F|_\Gamma=ASL(1,5)$ and $F|_\Delta=S_\Delta$.
	\item $k=0$, $n=6$, $F=PGL(2,5)$, $G=S_\Lambda$ where $\Lambda\subset \Omega$ and $|\Delta|=3$.
	\item $k=0$, $n=6$, $F=PGL(2,5)$, $G$ is cyclic of order 6 generated by a 3-cycle and a 2-cycle.
	\end{enumerate}
\end{enumerate}
\end{thm}	

The following fact will allow us to fully describe the exact factor $F$ for an exact factorization $S_n=FG$.

\begin{thm}\label{DM}\cite[Theorem 1.6C]{DM}
Let $F\leq Sym(\Omega)$, $\Gamma \neq \emptyset, \Omega$ a stable block and $\Delta = \Omega - \Gamma$.  If $F|_{\Gamma}$ and $F|_{\Delta}$ have no nontrivial homomorphic image in common, then $G=F|_{\Gamma}\times F|_{\Delta}$. \end{thm}

With this theorem in hand, we now give a better description of $F$.  

\begin{cor} If $F|_{\Delta}=1$, then $F=F|_{\Gamma}$. \end{cor}

\begin{cor}\label{AltTran} Suppose $|\Gamma|\geq 5$ and $F|_{\Gamma}=A_\Gamma$.  If $F|_\Delta \cong C_2$ then $F= A_\Gamma \times \lan \tau \ran$ for an involution $\tau\in F|_\Delta$ .  \end{cor}
\begin{proof}
The only non-trivial homomorphic image of $F|_\Delta \cong C_2$ is $C_2$.  Since $| \Gamma |\geq 5$, $A_\Gamma$ is simple and any homomorphism $A_\Gamma = F|_\Gamma \to C_2$ is trivial.  The result now follows. \end{proof}

The following result clarifies what $F$ is in part (2) of \Cref{WW}.

\begin{thm}\label{WWfix}  Let $S_n=FG$ be an exact factorization with $n\geq 7$.  Using the notation of \Cref{WW}, assume that $k=2$ or $k=3$ and $F|_\Gamma \cong S_\Gamma$.  If $F\neq S_\Gamma$ and $G$ is sharply $k$-transitive, then $F\cong S_\Gamma$ and $$F= \lan (1,j)(n-1,n) \ran:=\widetilde{S}_\Gamma \; \mbox{for}\; 1\neq j \in\Gamma.$$ \end{thm}

\begin{proof}
Since $S_n = FG$ is an exact factorization with $G$ is sharply k-transitive, $|G| = \frac{n!}{(n-k)!}$ and thereby $|F|=(n-k)!$.  However, as $$|F| = |F_\Gamma|\cdot |\left(F|_\Gamma\right) | = |F_\Gamma|(n-k)! ,$$ we have that $F_\Gamma = ()$.  Thus $F=F|_\Gamma$ and since $F\neq S_\Gamma$, it follows from \Cref{DM} that $F|_\Delta \neq ()$.  Thus $F_\Delta$ is a non-trivial normal subgroup of $F$.   Since $n\geq 7$, the prime in \Cref{WW} is at least 5 and so $|\Gamma | \geq 5$.  Thus $F_\Delta$ must be $A_\Gamma$.  Thus $F= A_\Gamma\lan \tau \ran$ where $\tau\in F$ is any involution that does not fix $\Delta$.  Since $\tau$ cannot fix $\Gamma$, we may assume $\tau = (1,2)(a,b)$. \\
Conjugating $F$ if necessary, we may further assume that $\Gamma = \{ 1,2,\ldots, n-k \}$, $a=n-1$ and $b=n$.  Since $A_\Gamma = \lan \sigma_i \ran\subset F$ for $\sigma_i = (i, n-k-1, n-k)$, $1\leq i \leq n-k-2$, we have for any $1\neq j\in \Gamma$
$$[(1,2)(n-1,n)]^{\sigma_2 \sigma_j\inv} = [(1,n-k-1)(n-1,n)]^{\sigma_j\inv} = (1,j)(n-1,n) \in F.$$
 Hence $\lan (1,i)(a,b) \ran = \widetilde{S}_{\Gamma} \subseteq F$ and thus $F=\widetilde{S}_\Gamma$. \end{proof}

We now see that for $n \geq 7$, by \Cref{WW}, that the factor $F$ is with finitely many exceptions one of the following: 
\num
\begin{equation}\label{Fknown}
F\in \{A_\Gamma, A_\Gamma \times C_2, S_\Gamma, \widetilde{S}_\Gamma\, S_n\times C_2\}.
\end{equation}

where $C_2$ is generated by a transposition that fixes $\Gamma$ and $\widetilde{S}_\Gamma$ is the group described in \Cref{WWfix}.  Note this last possibility of $F$ can only occur when $k\geq 2$.

\section{Frobenius-Schur Indicators and Hopf Algebras}
Again, we fix an algebraically closed field $\k$ for characteristic $0$.  We will use \cite{I} as a reference for the theory of finite group representations.  For the representation theory of symmetric groups, see \cite{J}. 

For $\k=\BC$, we have the famous Frobenius-Schur theorem.

\begin{thm}[Frobenius-Schur]\label{FSG} Let $\chi\in \mbox{Irr}(G)$.  For an integer $m\geq 2$, let $\nu_m(\chi)=(1/|G|)\sum_{g\in G} \chi(g^m) $ and $\chi^{(m)}$ be the class function defined by $\chi^{(m)}(g)=\chi(g^m)$ for a class function $\chi$. Then
\begin{enumerate}
\item $\chi^{(2)}$ is a difference of characters.
\item $\nu_2(\chi) =$ 1, -1 or 0.  Moreover,
    \begin{enumerate}
    \item $\nu_2(\chi) = 0$ if and only if $\chi$ is not real valued.
    \item $\nu_2(\chi) = 1$ if and only if $\chi$ is the character of a real representation.
    \item $\nu_2(\chi) = -1$ if and only if $\chi$ is real valued, but is not the character of a real representation.
    \end{enumerate}
\end{enumerate}
\end{thm}

\begin{thm}\label{SNI}
The irreducible characters of $S_n$ are realized over $\mathbb{Q}$, in particular $S_n$ has indicator values of $\nu_2(\chi)=1$ for all $\chi\in \mbox{Irr}(S_n)$.  Thus, $S_n$ is totally orthogonal.
\end{thm}

The following proposition was conveyed to us by Bob Guralnick.

\begin{prop}\label{Bob} Let $G$ be a finite group, $\k$ an alg. closed field of characteristic not 2.  Let $N$ be a subgroup of $G$ of index 2.  Suppose that every irreducible $\k G$-module has Frobenius-Schur indicator 1.  Then an irreducible $\k N$-module has Frobenius-Schur indicator 1 or 0. \end{prop}

\begin{proof} Let $W$ be a self dual irreducible $\k N$-module.  If $W$ extends to an irreducible $\k G$-module, then $G$ and $N$ have the same Frobenius-Schur indicator and so it is 1. 

If $W$ does not extend to $G$, then $V:= W_N^G = W_1 \oplus W_22$ as $N$-modules where $W_2$ is a twist
of $W_1$ and $V$ is irreducible as $G$-module (otherwise an irreducible constituent would have to
restrict to a twist of $W$ with respect to $N$).   By Frobenius reciprocity,  $W_1$ and $W_2$ are not
isomorphic as $\k N$-modules.   Thus, $W_2$ is not isomorphic to $W_1^*$ (because $W_1^*$ is isomorphic
to $W_1$).   

Now $V$ is self dual as a $G$-module with Frobenius-Schur indicator 1.   Since $W_1$ is irreducible for $N$, it is either totally singular or non-degenerate.   If it is totally singular, then a complement
for it would be isomorphic to its dual, but $W_1$ is the only $\k N$-submodule isomorphic to $W_1$ (and
$W_1^*$).   So $W_1$ is nonsingular.   Thus, the Frobenius-Schur indicator for $N$ on $W_1$ is the same as for $G$ on $V$ (i.e. 1). \end{proof}

\begin{cor}\label{ANI} The indicator value for an irreducible character of $A_n$ is $1$ or $0$. \end{cor}


Our main reference for Hopf algebras is \cite{Mo}.  We assume the reader is familiar with the notions of algebras and coalgebras.  We review our main object of focus, the bismash product Hopf algebra, $\k^G \# \k F$ constructed from an exact factorization of groups (equiv. a matched pair of groups).  Details can be found in \cite{Ma} and \cite{T}.

Let $(F,G,\rtri, \ltri)$ be a matched pair of groups.  We have induced actions $\rharp: F\times \k^G \to \k^G$ and $\lharp: G\times \k^F \to \k^F$ given by $$a\lharp p_x= p_{x\ltri a\inv },$$
$$x\rharp p_a = p_{x\rtri a},$$
for all $a\in F$, $x\in G$.  Moreover, these actions are algebra automorphisms.  These actions will be used to construct the multiplication and comultiplication on $H$ respectively.  Together with an appropriate antipode map, we construct the Hopf algebra $H=\k^G \# \k F$. 

\begin{df} \rm{ Let $(F,G, \rtri, \ltri)$ be a matched pair of groups.  Let $H=k^G \otimes \k F$ be the $\k$-vector space with basis $\{p_x \# a : x\in G, a\in F\}$.  Define multiplication by
$$(p_x\# a)(p_y\# b) = p_x(a\rharp p_y) \# ab$$ and comulitplication by $$\Delta(p_x\# a) = \sum_{y\in G} p_{x\inv{y}}\#(y\rtri a) \otimes p_y\# a.$$  With counit $\epsilon(p\# a)= \delta_{x,1}$ and antipode $S(p_x\# a) = p_{(x\ltri a)^{-1}} \# (x\rtri a)^{-1}$, we have a Hopf algebra structure on $H$. } \end{df}

In view of \Cref{EFconj}, one may ask how much these Hopf algebras $\k^G\# \k F$ depend upon the choice of exact factors. We provide an answer to this question in \Cref{CFHIT}, which states that these Hopf algebras are independent of the choice of conjugate subgroups.

\begin{prop}\label{AutHopf} \ Let $\phi$ be an automorphism of $L$ and $L=FG$ an exact factorization with actions $(\rtri, \ltri)$.  Then $L= \phi(F) \phi(G)$ is an exact factorization with actions $(\rtri_\phi, \ltri_\phi)$ and the Hopf algebras $H=k^G \# \k F$ and $H^\phi = \k^{\phi(G)} \# \k\phi(F)$ are Hopf isomorphic. \end{prop}

\begin{proof}
Note that $$\phi(g\rtri f) \phi(g\ltri f)=\phi(gf) = \phi(g)\phi(f) = \left( \phi(g)\rtri_\phi \phi(f) \right) \left( \phi(g) \ltri_\phi \phi(f) \right).$$
Thus the actions induced from the factorizations $L=FG$ and $L=\phi(F)\phi(G)$ satisfy 
$$\phi(g\rtri f) = \left( \phi(g)\rtri_\phi \phi(f) \right),$$
$$\phi(g\ltri f) = \left( \phi(g) \ltri_\phi \phi(f) \right).$$
Now the map $\varphi : H=\k^G \# \k F \to \k^{\phi(G)} \# \k\phi(F)$ defined by $p_x\# a \mapsto p_{\phi(x)} \# \phi(a)$ is readily verified to be a Hopf isomorphism.
\end{proof}

To show that conjugation of the factors of the exact factorization $L=FG$ yield no ``new"' Hopf algebras, we make some observations concerning the group actions.  These observations are of critical importance to show the proposed maps are, in fact, isomorphisms of Hopf algebras.

\begin{lem}\label{FXact} Let $L=FG$ be an exact factorization of a group $L$ with actions $(\rtri, \ltri)$.  Let $x\in L$ be written $x=fg$ for $f\in F$ and $g\in G$.  Then the actions $(\rtri_x, \ltri_x)$ induced from the factorization $L=F^x G$ satisfy 
$$b\rtri_x a^x = (b^{g\inv} \rtri a^f )^g, $$
$$b\ltri_x a^x = (b^{g\inv} \ltri a^f )^g, $$ 
for any $a\in F$ and $b\in G$.  Moreover, we have 
$$b^g \rtri_x a^g = \left( b\rtri a\right)^g,$$
$$b^g \ltri_x a^g = ( b\ltri a )^g, $$
\end{lem}

\begin{proof}
Let $a\in F$ and $b\in G$.  Then
\begin{align*}
(b \rtri_x a^x)(b \ltri_x a^x) = ba^x &=b^{g\inv g}a^{fg} \\
  &= \left(b^{g\inv} a^f\right)^g \\
  &= \left( b^{g\inv} \rtri a^f \right)^g \left( b^{g\inv} \ltri a^f \right)^g
\end{align*}
Since $F^x = F^{fg} = F^g$, the first pair of claims follow since $\left( b^{g\inv} \ltri a^f \right)^g\in F^g$ and $\left( b^{g\inv} \rtri a^f \right)^g \in G$.  For the remaining claims, note that $a^g = \left( a^{f\inv} \right)^x$ and the rest follows.
\end{proof}

\begin{prop}\label{FXHopf}  For the exact factorizations $L=FG$ and $L=F^x G$ for $x\in L$, the induced Hopf algebras $H=\k^G \# \k F$ and $H^{1,x} = k^G \# \k F^x$ are Hopf isomorphic.
\end{prop}
\begin{proof}
Let $x=fg$ for $f\in F$, $g\in G$.  Let $a,b\in F$ and $y,z\in G$.  Then $F^x=F^g$ and define $\varphi: H \to H^x$ by $\varphi(p_y \# a) = p_{y^g} \# a^g$.  A straight forward check shows that $\varphi$ is an isomorphism of Hopf algebras. 
\end{proof}

\begin{lem}\label{GXact} Let $L=FG$ be an exact factorization of a group $L$ with actions $(\rtri, \ltri)$.  Let $x\in L$ be written $x=gf$ for $f\in F$ and $g\in G$.  Then the actions $(\rtri_x, \ltri_x)$ induced from the factorization $L=FG^x$ satisfy 
$$b^x\rtri_x a = \left( b^g \rtri a^{f\inv} \right)^f $$
$$b^x\ltri_x a = \left( b^g \ltri a^{f\inv} \right)^f $$ 
for any $a\in F$ and $b\in G$.  Moreover, we have
$$b^f \rtri_x a^f = \left( b\rtri a \right)^f$$
$$b^f \ltri_x a^f = \left( b\ltri a\right)^f.$$
\end{lem}

\begin{proof}
Similar to \Cref{FXact}.
\end{proof}

\begin{prop}\label{GXHopf}  For the exact factorizations $L=FG$ and $L=FG^x$ for $x\in L$, the induced Hopf algebras $H=\k^G \# \k F$ and $H^{x,1} = \k^{G^x} \# \k F$ are Hopf isomorphic.
\end{prop}
\begin{proof}
Let $x=gf$ for $f\in F$, $g\in G$.  Let $a,b\in F$ and $y,z\in G$. Then $G^x=G^f$ and define $\varphi: H \to H^x$ by $\varphi(p_y \# a) = p_{y^f} \# a^f$.  A straight forward check shows that $\varphi$ is an isomorphism of Hopf algebras.
\end{proof}

\begin{thm}\label{CFHIT} Let $L=FG=F^yG^x$ be exact factorizations of a group $L$ and $x,y\in L$.  Then the Hopf algebras $H=\k^G\# \k F$ and $H^{x,y}=\k^{G^x} \# \k F^y$ induced from the factorizations are Hopf isomorphic. \end{thm}

In \cite{KMM}, the authors give a complete description of simple modules for a bismash product $H=\k^G\# \k F$. 

\begin{prop}\label{modules}\cite[Lemma 2.2 and Theorem 3.3]{KMM} Let $H=\k^G\# \k F$ be a bismash product as above.  For the action $x\ltri a$ of $F$ on $G$, fix one element $x$ in each $F$-orbit of $G$ and let $F_x$ be the stabilizer of $x$ in $F$.  Let $V$ be a simple left $F_x$-module and let $\widehat{V}= \k F\otimes_{\k F_x} V$ denote the induced $\k F$ module. \\
Then $\widehat{V}$ becomes an $H$-module via $$(p_y\# a) [ b\otimes v] = \delta_{y\ltri(ab), x}\, (ab\otimes v)$$ for $y\in G$, $a,b\in F$ and $v\in V$.  Moreover, $\widehat{V}$ is a simple $H$-module under this action and every simple $H$-module arises in this way. \end{prop}


Next we focus on computing Schur indicators for those Hopf algebras which arise from exact factorizations of the symmetric group.  Most of our compuations will be done using \cite{GAP} with custom programs.  Our main result is an extension of \cite[Theorem 3.6]{JM}, which stated the Hopf algebra $H_n = k^{C_n} \# k S_{n-1}$ arising from the factorization of $S_n = S_{n-1}\cdot C_n$ with $C_n = \lan (1,2,\ldots, n) \ran$, is totally orthogonal; that is, $H_n$ has Schur indicator values of 1 for all simple modules.  Our result shows the same is essentially true when $S_n = S_{n-k} \cdot G$ is an exact factorization for $1\leq k \leq 5$.  When $k=1$, the Hopf algebra $H = \k^G \# \k S_{n-1}$ is totally orthogonal.  When $2\leq k \leq 5$, the irreducible modules have Schur indicator values of 1 or 0.  However, we shall describe exactly when an irreducible module has indicator 0.

We begin by defining the second Frobenius-Schur indicators and then recalling the extension of the Frobenius-Scur theorem of groups to semisimple Hopf algebras.

\begin{df} \rm{ Let $H$ be a semisimple Hopf algebra and $\Lambda \in \int^H$ with $\epsilon(\Lambda)=1.$   Let $\Lambda^{[2]}= m(\Delta(\Lambda))$, i.e., $\Lambda^{[2]}=\sum_{\Lambda} \Lambda_1 \Lambda_2$.  For $\chi\in\mbox{Irr}(H)$, let $\nu_2(\chi) = \chi(\Lambda^{[2]})$ be the \textbf{Frobenius-Schur indicator of} $\chi$. } \end{df}

If $V$ is an irreducible $H$-module, write $\chi_V$ to be the corresponding character and similarly $V_\chi$ will be corresponding module for a character $\chi$.  The following theorem generalizes the Frobenius-Schur theorem for semisimple Hopf algebras.

\begin{thm}\label{FSH}\cite[2000]{LM}
Let $H$ be a semisimple Hopf algebra over an algebraically closed field $k$.  If $k$ has characteristic $p\neq 0$, assume in addition that $p\neq 2$ and $H^*$ is semisimple.  Then for $\Lambda$ and $\nu_2(\chi)$ as above, and $\chi \in Irr(H)$, the following properties hold:
\begin{enumerate}
\item $\nu_2(\chi) = 0, 1 $ or $-1$, $\forall\; \chi\in Irr(H)$, 
\item $\nu_2(\chi) \neq 0$ if and only if $V_\chi \cong V_\chi^*$.  Moreover $\nu_2(\chi)=1$ (respectively $-1$) if and only if $V_\chi$ admits a symmetric (resp. skew-symmetric) non-degenerate bilinear $H$-invariant form. 
\item Considering $S\in End(V)$, $Tr S = \sum_{\chi\in Irr(H) }\; \nu_2(\chi)\cdot \chi( 1_H)$.
\end{enumerate}
\end{thm}

In order to compute the Schur indicators, we first need to know what the second Sweedler power of the normalized integral element $\Lambda = \frac{1}{|F|}\sum_{a\in F} p_1 \# a$ is.  Now we see

\num 
\begin{equation}\Delta(\Lambda) = \frac{1}{|F|} \sum_{a\in F} \Delta( p_1\# a)  = \frac{1}{|F|}\sum_{a\in F} \sum_{y\in G} p_y\# (y\inv \rtri a) \otimes p_{y\inv }\# a.\end{equation}

Multiplying the terms on opposite sides of the tensor yields $\Lambda^{[2]}$,

\num
\begin{equation}\Lambda^{[2]} = \frac{1}{|F|} \sum_{a\in F}\sum_{y\in G} \delta_{ y\inv ,\, y\ltri(y\inv \rtri a)}\, p_y \# (y\inv \rtri a)a.\end{equation}

In terms of computations, this form is not the best.  In \cite{JM}, the authors give a less computationally expensive form.

\begin{df} \rm{ For $x\in G$, let $$F_{x,x\inv} = \{ a\in F : x\ltri a = x\inv \}.$$ } \end{df}

Using the observation that $(y\inv \ltri a )\inv = y \ltri (y\inv \rtri a)$ \cite[Lemma 4.2]{JM}, the authors give 
\num
\begin{equation}\label{FSL}
\Lambda^{[2]} = \frac{1}{|F|} \sum_{y\in G} \sum_{a\in F_{y\inv, y} }p_y \# (y\inv \rtri a)a.
\end{equation}

Also in \cite{JM}, the authors give a convenient formula for computing the induced characters $\widehat{\chi}$ of the induced $F_x$ character $\chi$.

\begin{lem}\cite[Lemma 4.6]{JM}\label{JM4.6}  The character $\widehat{\chi}$ of $\widehat{V}$ may be computed as follows: $$\widehat{\chi} (p_y\# a) = \sum_{b\in T_x} \delta_{y\ltri b, x} \chi(b\inv a b),$$
for any $y\in G$, $a\in F$, where the sum is over all $b$ in a fixed set $T_x$ of representatives of the right cosets of $F_x$ in $F$, and $\chi$ is the character of $V$. \end{lem}

Combining this with equation \Cref{FSL}, the authors show

\num
\begin{equation}\label{FSIform}
\nu_2(\widehat{\chi}) = \widehat{\chi}(\Lambda^{[2]}) = \frac{1}{|F|} \sum_{y\in \mathcal{O}_x} \sum_{a\in F_{y\inv , y}} \widehat{\chi} (p_y\# (y\inv\rtri a)a).\end{equation}

\begin{lem}\label{orthog} For any semisimple Hopf algebra $H$ over an algebraically closed field $k$ of characteristic not 2, $Tr(S)= \sum_{\chi} \chi(1) $ if and only if $H$ is totally orthogonal.\end{lem}
\begin{proof}
This is immediate from \Cref{FSH}, since $\nu(\chi)=0,\,1$ or $-1$. \end{proof}

\begin{df} \rm{ For a group $L$, let $i_L := | \{ l\in L : l^2=1 \} |$. If $L=S_n$, then simply write $i_n :=i_{S_n}$.} \end{df}

\begin{lem}\label{JM2.10} Let $L=FG$ be a factorizable group and $S$ the antipode of the bismash product $H=\k^G\# \k F$ as above.  Then $Tr(S)= i_L$. \end{lem}
\begin{proof}
\cite[Lemma 2.10]{JM}.\end{proof}

An interesting situation may occur for bismash products of Hopf algebras.  It is possible to have $x\in G$ such that \textbf{every} induced module from the orbit of $x$ has Schur indicator 0.  We given such orbits a special moniker and state some equivalent properties.

\begin{df} \rm{ For $x\in G$, we say the orbit $\mathcal{O}_x$ is a \textbf{null indicator orbit} if every irreducible module $\widehat{V}$ of the Hopf algebra $H=\k^G \# \k F$ induced from an irreducible $F_x$-module $V$ has Schur indicator equal to 0.}
\end{df}

\begin{prop}\label{NIO} The following are equivalent for $x\in G$:
\begin{enumerate}
\item $\mathcal{O}_x$ is a null indicator orbit,
\item $x\inv \notin \mathcal{O}_x$,
\item $y\inv \notin \mathcal{O}_x$ for some $y\in \mathcal{O}_x$,
\item $y\inv \notin \mathcal{O}_x$ for every $y\in \mathcal{O}_x$.
\end{enumerate}
\end{prop}

\begin{proof}
The equivalence of (2), (3) and (4) is proven by the equivalences of \cite[Corollary 4.3]{JM}. 

Assume $x\inv \in \mathcal{O}_x$.  Consider the trivial character $\chi$ of $F_x$ and its induced character $\widehat{\chi}$.  Thus, $\widehat{\chi}(p_y\# a)$ is 0 or 1 by the definition of $\chi(b\inv a b)$ in \Cref{JM4.6}.  To show $\nu_2(\widehat{\chi}) \neq 0$, it suffices to to show one term of \Cref{FSIform} is non-zero.  If we take $y=x$, since $x\in \mathcal{O}_x$, $F_{x\inv,x} \neq \emptyset$ by assumption.  Thus \Cref{FSIform} has terms of the form:

\begin{equation*}
\sum_{a\in F_{x\inv , x}} \widehat{\chi} (p_x\# (x\inv\rtri a)a).
\end{equation*}
By \Cref{JM4.6}, we may take $b=1$ for one of the representatives and thus the above simplifies to $$\sum_{a\in F_{x\inv, x}} \chi \left((x\inv \rtri a)a\right).$$
To show this term in non-zero, we show that $(x\inv \rtri a) a\in F_x$.  Using the properties of \Cref{MPA}, 
\begin{align*}
x\ltri \left( (x\inv \rtri a) a\right) &= \left( x\ltri (x\inv \rtri a) \right) \ltri a \\
&= \left( x\inv \ltri a\right)\inv \ltri a \\
&= (x\inv)\ltri a \\
&= x
\end{align*}
This establishes that (1) implies (2).

Assume (4) holds.  Then for every $y\in \mathcal{O}_x$, $F_{y\inv, y} = \emptyset$.  Then \Cref{FSIform} is zero for any irreducible module $V$. 
\end{proof}

\begin{prop}\label{stabFNIO} Let $L=FG$ be an exact factorization and $H=\k^G\# \k F$ be the bismash product Hopf algebra.  If $x\in G$ has stabilizer $F_x=F$, then:
\begin{enumerate}
\item If $x=1$, then the value of the Schur indicator of $\widehat{V}$ is the same as the indicator for $V$ as a simple $\k F$ module.
\item $x^2 \neq 1$ if and only if $\mathcal{O}_x$ is a null indicator orbit.
\end{enumerate}
\end{prop}
\begin{proof}
This is the content of \cite[Proposition 4.9]{JM} in the phrasing of this paper. \end{proof}

As we see in the next section, the number of involutions in $S_n$ will of great importance.  Letting $i_n$ denote the number of involutions in $S_n$, we have the well known recursive formula:

\num
\begin{equation}
i_n = i_{n-1} + (n-1)i_{n-2}
\end{equation}
\section{The Structure of Simple Modules when $G$ is Sharply $k$-transitive}
For an exact factorization $S_n = S_{n-k}\cdot G$ with $S_{n-k}$ fixing $\{n-k+1, \ldots, n\}$ and $G$ sharply $k$-transitive, we first develop a way of identifying elements of $G$.  With this identification, we compute $F$-stabilizers.  Next, we shall gather similar $g\in G$ into \textbf{members}, and then group members together into \textbf{families}.  We shall then prove that the members are $F$-orbits.  Using this language, we will easily see which orbits are null indicator orbits.  In lieu of \Cref{WW}, we see that $1\leq k \leq 5$ and $k < n/2$.  To prove the main result in each subsection, we argue as follows: First, establish the number of $F$-orbits using Burnside's lemma.  Second, identify which, and how many, of these orbits are null indicator orbits.  Finally, we compute the trace of the antipode and see that we have exactly enough irreducibles so that the non-zero indicators are forced to be 1.

With a firm hand on the orbits and stabilizers, we can then develop a general method to prove out main result for each $k$, $1\leq k \leq 5$.  In this section, $\Gamma = \{1,2,\ldots, n-k \}$ and $\Delta = \{n-k+1, n-k+2, \ldots, n\}$.

\begin{lem}\label{Gident} Let $G$ be sharply $k$-transitive on $\Omega = \{ 1, \ldots, n\}$.  We have a 1-1 correspondence between $G$ and duplicate-free tuples $[\alpha_1,\ldots, \alpha _k] \in \Omega^k$.  We then write $g\simeq [\alpha_1,\ldots,\alpha_k]$. \end{lem}
\begin{proof}
For $g\in G$, associate with $g$ the $k$-tuple $[\alpha_1, \ldots, \alpha_k]$ with $\alpha_i = (n-k+i)^g$.  Since $G$ is sharply $k$-transitive, such a $g$ exists and is unique for every tuple $[\alpha_1, \ldots, \alpha_k]$.
\end{proof}

\begin{prop}\label{stabdesc} Let $S_n=FG$ be an exact factorization with $F=S_{n-k}$ fixing $\{n-k+1, \ldots, n\}$ and $G$ sharply $k$-transitive.  For $g\in G$, $$F_g = \{f\in F : \alpha_i ^ f = \alpha_i, \,\; 1\leq i \leq k\},$$ where $g\simeq [\alpha_1, \alpha_2, \ldots, \alpha_k]$.\end{prop}
\begin{proof}
Let $1\leq i \leq k$ and suppose $f\in F$ fixes $\alpha_i\in \Omega$.  Since $g\rtri f \in F$, $g\rtri f$ will fix $n-k+i\in \Omega$ and so $$\alpha_i = \alpha_i^f = (n-k+i)^{gf} = (n-k+i)^{(g\rtri f)(g\ltri f)} = (n-k+i)^{g\ltri f}.$$
Since $g$ is sharply $k$-transitive, $g\ltri f =g$.  Conversely, suppose $f\in F_g$.  Then $f = g\inv (g\rtri f) g$ and so 
$$(\alpha_i)^f = (\alpha_i)^{g\inv (g\rtri f) g} = (n-k+i)^{(g\rtri f) g} = (n-k+i)^g = \alpha_i.$$
\end{proof}

\begin{df} \rm{Let $G$ be sharply $k$-transitive. For $g\in G$, $g\simeq [\alpha_1, \alpha_2, \ldots, \alpha_k]$, the \textbf{$\Gamma$-indices of $g$}, $I(g)$, is the subset of indices $$I(g) = \{ i : \alpha_i \in \Gamma\}.$$ } \end{df}

\begin{df} \rm{ Let $G$ be sharply $k$-transitive. For $g\simeq [\alpha_1, \alpha_2, \ldots, \alpha_k]$, the \textbf{ $\Delta$-tuple of $g$}, $D(g)$, is the tuple of elements $$D(g) = [\alpha_j],  j\notin I(g).$$  The \textbf{$\Delta$-set} of $g$, $\Delta(g)$ is the subset of $\Delta$ $$\Delta(g) := \{\alpha_j : j\notin I(g) \}.$$ } \end{df}

\begin{ex} \rm{ Let $S_{12} = S_9 \cdot PGL(2,11)$ with $$PGL(2,11) =  \lan (1,2,3,4,5,6,7,8,9,10,11,12), (1,3)(2,7)(4,8)(5,9)(6,11)(10,12) \ran.$$  Then $G$ is sharply 3-transitive, $\Gamma = \{1,\ldots, 9\}$ and $\Delta=\{10,11,12\}$.
\begin{center}
\begin{tabular}{c|c|c|c|c}
          $g$                  &$[\alpha_1, \alpha_2, \alpha_3]$&     $I(g)$       &   $D(g)$    & $\Delta(g)$\\
          \hline
$(1,10,9,5,12)(2,11,8,4,3)$      &          $ [9, 8, 1]     $               &     $\{ 1,2,3\}$ & $\emptyset$ & $\emptyset$ \\
$(1,2,3,4,5,6,7,8,9,10,11,12) $  &          $ [11, 12, 1]  $                &    $\{3\}$       & $[11,12]$   & $\{11,12\}$ \\
$(1,7,8)(2,4,5)(3,9,6)(10,12,11)$&          $ [12, 10, 11] $                &   $\emptyset$    & $[12, 10, 11]$& $\{10,11,12\}$ \\
$(1,7,4,5,3)(2,12,10,11,8)$      &          $ [11, 8, 10] $                 &   $\{ 2\}$       & $[11,10]$   & $\{10,11\}$ \\
\end{tabular}
\end{center}
}
\end{ex}

\begin{df} \rm{ Let $G$ be sharply $k$-transitive.  For a subset of indices $I\subseteq \{1,2,\ldots k\}$ and tuple $D$ of $k-|I|$ many distinct elements of $\Delta$, the \textbf{$(I,D)$-member}, $\mathcal{M}_{I,D}$, is $$\mathcal{M}_{I,D} = \{ g : I(g)=I, D(g)=D \}.$$ 
We say two members $\mathcal{M}_{I,D}, \mathcal{M}_{I', D'}$ are \textbf{the same type}  if $I=I'$.  We say two members of the same type, $\mathcal{M}_{I,D}, \mathcal{M}_{I,D'}$, are \textbf{permutation related} if for any $g\in \mathcal{M}_{I,D}$, $g'\in \mathcal{M}_{I,D}$, $\Delta(g)=\Delta(g')$.  } \end{df}

\begin{df} \rm{ We say a member $\mathcal{M}_{I,D}$ is \textbf{unmixed} if for all $i\in I$, $$n-k+i \notin \Delta(g)$$ for some $g\in \mathcal{M}_{I,D}$.  We say the member is \textbf{mixed} otherwise.  } \end{df}

Since $\Delta(g)=\Delta(g')$ for all $g,g' \in \mathcal{M}_{I,D}$, the choice of $g$ is irrelevant.  Another way to describe an unmixed member is the following:  Let $x \in \Delta$.  Then for any $g\in \mathcal{M}_{I,D}$ with $x^g \in \Gamma$, we have $x\notin \Delta(g)$.  This notion will be important when we determine which orbits are null indicator orbits.

\begin{df}\label{family} \rm{ Let $G$ be sharply $k$-transitive.  The \textbf{$m$-family} is the set $$\{ \mathcal{M}_{I,D} : |I|=m\}.$$  We simply call an element of a $m$-family a \textbf{member}. } \end{df}

Notice that every $g\in G$ belongs to exactly one $\mathcal{M}_{I,D}$ and each $\mathcal{M}_{I,D}$ belongs to exactly one $m$-family $\mathcal{F}_m$.

\begin{lem} The $(I,D)$-member $\mathcal{M}_{I,D}$ contains $\binom{k}{|I|}\cdot (|I|)!$ many $g\in G$.  For a given $m$-family, there are $\binom{k}{m}$ many member types. \end{lem}
\begin{proof}
For $g \simeq [\alpha_1,\ldots,\alpha_k]$  with $I(g)=I$ and $D(g)=D$,  $g\in \mathcal{M}_{I,D}$ if and only if $\alpha_i \in \Gamma$ for $i\in I$.  There are exactly $\binom{k}{|I|}\cdot (|I|)!$ many $g$ with these properties. 

In a given $m$-family, a member type is determined by $\{ i : \alpha_i \in \Gamma\}$.  There are exactly $\binom{k}{m}$ such subsets. \end{proof}

\begin{prop}\label{orbdesc} Let $S_n=FG$ be an exact factorization with $F=S_{n-k}$ fixing $\{n-k+1, \ldots, n\}$ and $G$ sharply $k$-transitive.  For any $g\in G$, the $F$-orbit of $g$ $\mathcal{O}_g = \mathcal{M}_{I(g),D(g)}$. \end{prop}
\begin{proof}
For $g\simeq [\alpha_1,\ldots, \alpha_k]$, let $I:= I(g)$ and note that $[F : F_g] = \binom{k}{|I|}\cdot (|I|)! = | \mathcal{O})_g |$.  Indeed, since the number of $f\in S_{n-k}$ that fix $|I|$ many points in $\Gamma$ is $(n-k-|I|)!$.  Since this is exactly the size of $\mathcal{M}_{I(g),D(g)}$, it will suffice to show one containment. 

Let $\{ \beta_i \}$,  be $|I|$ many distinct elements of $\Gamma$.   Choose $f\in F=S_{n-k}$ such that $\alpha_i ^ f = \beta_i$ for all $i\in I$.  Since $m\leq k \leq 5 < n/2$, such $f$ exists.  Then for $i\in I$ and $j\notin I$,
$$\beta_i = \alpha_i^f = (n-k+i)^{gf} = (n-k+i)^{(g\rtri f)(g\ltri f)} = (n-k+i)^{g\ltri f},$$
$$\alpha_j = \alpha_j^f = (n-k+j)^{gf} = (n-k+i)^{(g\rtri f)(g\ltri f)} = (n-k+j)^{g\ltri f}$$ 
and so $I(g\ltri f) = I$ and $D(g\ltri f) = D(g)$.  Thus, $\mathcal{O}_g \subseteq \mathcal{M}_{I(g),D(g)}$. \end{proof}

\begin{cor}\label{stabfamcor} If $g$ belongs to a member of the $m$-family, then $F_g \cong S_{n-k-m}$.  Moreover, each $m$-family contains $\binom{k}{m}\cdot \binom{k}{k-m}\cdot(k-m)!$ members. \end{cor}
\begin{proof}
The statement about stabilizers is clear from \Cref{stabdesc} and \Cref{orbdesc}.  Since $m=|I|$ for any member $\mathcal{M}_{I,D}$ of the $m$-family, there are $\binom{k}{m}$ many choices of the subset $I$ and $\binom{k}{k-m}\cdot (k-m)!$ many choices of the $\Delta-tuple$ $D$. \end{proof}

We see that there is a 1-1 correspondence between $F$-orbits of $G$ and members $\mathcal{M}_{I,D}$.  We now distinguish and count which members correspond to orbits that are \textbf{not} null indicator orbits.  When $m=0$, the next theorem is just \Cref{stabFNIO}, and so it is an extension of this result to the factors under our consideration. 

\begin{thm}\label{NIOdesc} For a given $m$-family, the number of members which are \textbf{not} null indicator orbits is $\binom{k}{m} \cdot i_{k-m}$, where $i_{k-m}$ is the number of involutions in $S_{k-m}$  for $m<k$ and $i_0 := 1$. \end{thm}
\begin{proof}
By \Cref{NIO}, to show a member is a null indicator orbit, it suffices to show $g\inv \notin \mathcal{M}_{I,D}$ for some $g\in\mathcal{M}_{I,D}$.  We first show that any mixed member corresponds to a null indicator orbit.  

Suppose $\mathcal{M}_{I,D}$ is mixed and choose $g\in \mathcal{M}_{I,D}$ such that $x^g \in \Gamma$ and $x\in \Delta(g)$.  Thus, $y^g = x$ for some $y\in \Delta$.  Then $x^{g\inv} = y \notin \Gamma$, so $I(g\inv) \neq I$.  Thus, $\mathcal{M}_{I,D}$ corresponds to a null indicator orbit. 

Choose an subset $I$ of indices of size $0\leq m < k$ and consider all unmixed members of type $I$.  For each unmixed member $\mathcal{M}_{I,D}$, assign a permutation $\sigma_{I,D}$ in $S_{k-m}$ as follows: Using the map notation of a permutation, $$\sigma_{I,D} = \begin{pmatrix} n-k+j \\ \alpha_j \end{pmatrix}$$
where $j\in \{1,2,\ldots, k\} - I$ and $\alpha_j = (n-k+j)^g$ for any $g\in \mathcal{M}_{I,D}$.  Since every $g\in\mathcal{M}_{I,D}$ has the same $\Delta$-tuple, $\alpha_j$ is the same for each $g$ and the permutation is independent of the choice of $g$.  Since every member is of the same type and unmixed, the ``proposed" bijection is in fact a bijection.  Finally, since each member of type $I$ is uniquely determined by its $\Delta$-tuple, this assignment is 1-1.  

Note that for any $g\in \mathcal{M}_{I,D}$, the assignment above is simply is given by restriction of $g$, that is, $$\sigma_{I,D} = g|_{\Delta(g)}.$$
Thus we see if $g\in \mathcal{M}_{I,D}$, then $g\inv$ belongs to the member assigned to the permutation $\sigma_{I,D}\inv$.  So only the orbits $\mathcal{M}_{I,D}$ that correspond to involutions are \textbf{not} null indicator orbits. 

Since the choice of type $I$ was arbitrary and there are $\binom{k}{m}$ many member types in the $m$-family, there are exactly $\binom{k}{m}\cdot i_{k-m}$ many orbits that are \textbf{not} null indicator orbits.  

When $m=k$, we have a single orbit since $I=\{1,2\ldots, 5\}$ and $D=\emptyset$.  Let $g\in \mathcal{M}_{I,D}$.  Since each $x\in \Delta$ is mapped to something in $\Gamma$ under $g$, for each $x\in \Delta$, there is $\alpha_x\in \Gamma$ such that $\alpha_x^g = x$.  Hence, $g\inv \in \mathcal{M}_{I,D}$ and so this member is not a null indicator orbit.  By setting $i_0=1$, this result agrees with proposed formula. \end{proof}

\begin{df}\label{modulenumbers} \rm{ For each $n$, enumerate the simple $S_n$ modules by $\{1, \ldots, c_n\}$.  Denote by $d_{n, i}$, the dimension of the $i^{\mbox{\tiny{th}}}$ simple $S_n$ module.  } \end{df}

The following combinatorics result will be necessary to decide upon the number of $G$-orbits.

\begin{prop}\label{combine}
 Let $\phi(m)$ denote the number of permutations $\sigma \in S_n$ such that $\sigma$ has \textbf{exactly} $m$ fixed points.  Then
 \begin{enumerate}
 \item $\sum_{m=0}^n \phi(m) = n!$,
 \item $\sum_{m=0}^n m\phi(m) = n!$,
 \item $\sum_{m=0}^n m^2 \phi(m) = 2(n!)$, and
 \item $\sum_{m=0}^n m^3 \phi(m) = 5(n!)$.
  \end{enumerate}
 \end{prop}

\begin{proof}
The first equality is clear since the left hand side is simply a way of grouping and then counting the number of permutations of $S_n$. 

In the second equality, note that the sum is counting the number of ordered pairs $(i,\sigma)$ such that $i^\sigma = i$ for $i\in \{1,2,\ldots, n\}$, $\sigma \in S_n$.  Indeed, for if $\sigma$ has $m$ fixed points, then there are exactly $m$ such $(i, \sigma)$.  We now count the number of such pairs $(i,\sigma)$.  For each index $i$, there are exactly $(n-1)!$ many $\sigma$ with $i^\sigma = i$.  As there are $n$ choices for $i$, the equality is proven. 

In the third equality, the sum counts the number of ordered triples $(i,j,\sigma)$ such that $i^\sigma = i$ and $j^\sigma = j$.  Indeed, for if $\sigma$ has $m$ many fixed points, there are $m^2$ many tuples $(i,j,\sigma)$ with the desired property.  As before, we now proceed to count the number of such tuples.  For $i\neq j$, there are exactly $(n-2)!$ many $\sigma \in S_n$ that fix both $i$ and $j$.  There are $n(n-1)$ many ordered pairs $(i,j)$ with $i\neq j$ and so we have $n!$ many such $(i,j,\sigma)$.  When $i=j$, there are $(n-1)!$ many $\sigma$ and $n$ many such $i$.  This gives the sum is $n! + n! = 2(n!)$. 

The fourth equality, the sum counts the number of $(i,j,k,\sigma)$ such that $i,j,k$ are all fixed by $\sigma$.  If $i,j,k$ are all distinct, there are $(n-3)!$ many $\sigma$ and $n(n-2)(n-3)$ many tuples $(i,j,k)$ and so we have $n!$ many such tuples $(i,j,k,\sigma)$ with $i,j,k$ distinct.  If only two indices are distinct, there are $(n-2)!$ many permutations fixing two indices.  We have three possibilities for the tuples, $(i,i,j)$, $(i,j,i)$ or $(j,i,i)$ with $i\neq j$.  This gives $3(n)(n-1)$ many tuples $(i,j,k)$ with only two distinct entries and so we have $3(n!)$ many such tuples $(i,j,k,\sigma)$.  Finally, if all indices are the same, we have $n!$ many tuples $(i,j,k,\sigma$ as before.  Thus, the sum is equal to $n! + 3(n!) +n! = 5(n!)$. \end{proof}
When we have an exact factorization $S_n=FG$ with $k=1$ and so $F=S_{n-1}$ fixing $n$ and $G$ is sharply 1-transitive, i.e. regular on $\Omega$.  By \Cref{Gident}, each $g\in G$ is identified with a single point $g\simeq [i]$ where $n^g=i$.  We use the same notation of \Cref{modulenumbers} for the description of simple $S_n$ modules for any $n$. 

\begin{lem} For $g\simeq [i]$, $F_g = \{ f \in F : i^f=i\}$. \end{lem}
\begin{proof}
See \Cref{stabdesc} with $k=1$. 
\end{proof}

\begin{lem} The number of $F$-orbits is 2. \end{lem}

\begin{proof} Since $g\ltri f = g$ if and only if $i^f=i$ $g~[i]$, we see that if $f$ has $m$ fixed points as a permutation in $S_{n-1}$, then $f$ fixes exactly $(m+1)$ many pairs $[i]~g$.  Indeed for if $f$ has $m$ fixed points in $S_{n-1}$, then it has $m+1$ fixed points as a permutation in $S_{n}$.  If we let $N$ be the number of $F$-orbits and $\phi(m)$ be the number of $f\in F$ with $m$ fixed points, we see that

\begin{align*}
N  &= \frac{1}{(n-1)!} \sum_{\sigma\in F} \fix(\sigma) \\
   &= \frac{1}{(n-1)!} \sum_{m=0}^{n-1} \phi(m) (m+1) \\
   &= \frac{1}{(n-1)!} \sum_{m=0}^{n-1} \left( m \phi(m) +\phi(m) \right) \\
   &= \frac{1}{(n-1)!} \left[ (n-1)! + (n-1)! \right] \\
   &= 2,
\end{align*}
where the last equality follows from \Cref{combine}.
\end{proof}

\begin{lem} Neither of the two orbits are null indicator orbits. \end{lem}
\begin{proof}
We use \Cref{orbdesc} to describe the orbits.  As there are only two families, we have by \Cref{NIOdesc} that number of null indicator orbits is $$2 - \binom{1}{0}\cdot i_1 - \binom{1}{1} i_0 = 2-1-1=0.$$
\end{proof}

\begin{lem}\label{moddesck1} Let $H=\k^G \# \k F$ be the bismash product associated with the exact factorization $S_n=FG$.  The number of simple $H$-modules is $i_{n-1} + (n-1)i_{n-2}$.  The dimension of each module is contingent upon the $F$-orbits of $G$ as follows:  For $V$ the $i^{\mbox{\tiny{th}}}$ irreducible $F_x$ module, then for the induced $H$-module $\widehat{V}$:
\begin{center}
\begin{tabular}{r l}
 $\dimn (\widehat{V}) = d_{n-1,i}$ & if $F_x = S_{n-1}$ \\
 $\dimn (\widehat{V}) = (n-1)d_{n-2,i}$ & if $F_x = S_{n-2}$ \\
\end{tabular}
\end{center}
\end{lem}
\begin{proof}
Since $\k F$ is a free $\k F_x$-module of rank $[F : F_x]$, the statement about dimensions is clear.  Any irreducible $H$-module is induced from a $F_x$-module for a given set of $F$-orbit representatives.  Since there are exactly two orbits, the proof is complete. \end{proof}

\begin{prop}\label{kis1} Let $H=\k ^G \# \k F$ be the bismash product arising from an exact factorization of $S_n = S_{n-1} \cdot G$.  Then $H$ is totally orthogonal. \end{prop}
\begin{proof} 
Recall that $Tr(S) = \sum_{\chi} \nu^{[2]}(\chi)\chi(1)= i_n$ for the Hopf algebra $H$, and $i_n = \sum_{i=1}^{c_{n-2}} d_{n,i}$ since $S_n$ is totally orthogonal.  Now by our description of $H$-modules, by \Cref{moddesck1} and since $i_n = i_{n-1} + (n-1)i_{n-2}$, we must have $\nu^{[2]}(\chi)=1$ for all irreducible modules. \end{proof}
When we have an exact factorization $S_n=FG$ with $k=2$ and so $F=S_{n-2}$ fixing $n-1,n$ and $G$ is sharply 2-transitive on $\Omega$.  By \Cref{Gident}, each $g\in G$ is identified with an ordered pair $g\simeq [i,j]$ where $(n-1)^g=i,\; n^g=j$.  We use the same notation of \Cref{modulenumbers} for the description of simple $S_n$ modules for any $n$. 

\begin{lem} For $g \simeq [i,j]$, $$F_g = \{ \sigma \in F=S_{n-2} : i^\sigma = i, j^\sigma = j \}.$$
\end{lem}
\begin{proof}
\Cref{stabdesc} with k=2. \end{proof}

\begin{lem} For an exact factorization $S_n = F G$, the number of $F$-orbits in $G$ is 7. \end{lem}
\begin{proof}
Since $g\ltri f = g$ if and only if $i^f=i$ and $j^f=j$ for $g~[i,j]$, we see that if $f$ has $m$ fixed points as a permutation in $S_{n-2}$, then $f$ fixes exactly $(m+2)(m+1)$ many pairs $[i,j]~g$.  Indeed for if $f$ has $m$ fixed points in $S_{n-2}$, then it has $m+2$ fixed points as a permutation in $S_{n}$.  If we let $N$ be the number of $F$-orbits and $\phi(m)$ be the number of $f\in F$ with $m$ fixed points, we see that

\begin{align*}
N  &= \frac{1}{(n-2)!} \sum_{\sigma\in F} \fix(\sigma) \\
   &= \frac{1}{(n-2)!} \sum_{m=0}^{n-2} \phi(m) (m+1)(m+2) \\
   &= \frac{1}{(n-2)!} \sum_{m=0}^{n-2} \left( m^2 \phi(m) + 3m\phi(m) +2\phi(m) \right) \\
   &= \frac{1}{(n-2)!} \left[ 2(n-2)! + 3(n-2)! + 2(n-2)! \right] \\
   &= 7.
\end{align*}

\end{proof}

\begin{lem} There only exact 2 null indicator orbits, and these orbits belong to the $1$-family. \end{lem}
\begin{proof}
By \Cref{NIOdesc}, the number of orbits which are \textbf{not} null indicator orbits is $\binom{2}{0}i_2 + \binom{2}{1} i_1 + \binom{2}{2}i_0 = 2 + 2 + 1=5$. \\
Moreover, by \Cref{stabfamcor}, there are 4 orbits in the $1$-family and the proof is complete.
\end{proof}

\begin{prop}\label{moddesck2} The total number of simple $H=\k ^G \# \k S_{n-2}$ modules which may have non-zero indicator values is $2c_{n-2} + 2c_{n-3} + c_{n-4}$.  The dimension of each module is contingent upon the $F$-orbits of $G$ as follows:  For $V$ the $i^{\mbox{\tiny{th}}}$ irreducible $F_x$ module, then for the induced $H$-module $\widehat{V}$:
\begin{center}
\begin{tabular}{r l}
 $\dimn (\widehat{V}) = d_{n-2,i}$ & if $F_x = S_{n-2}$ \\
 $\dimn (\widehat{V}) = (n-2)d_{n-3,i}$ & if $F_x = S_{n-3}$ \\
 $\dimn (\widehat{V}) = (n-2)(n-3)d_{n-4,i}$ & if $F_x = S_{n-4}$ . \\
\end{tabular}
\end{center}
\end{prop}
\begin{proof}
Since $\k F$ is a free $\k F_x$-module of rank $[F : F_x]$, the statement about dimensions is clear.  Any irreducible $H$-module is induced from a $F_x$-module for a given set of $F$-orbit representatives.  Since exactly two of the orbits with stabilizer $S_{n-3}$ are null indicator orbits, the proof is complete. \end{proof}

\begin{lem}\label{involk2}
Let $i_n$ denote the number of involutions of $S_n$, then $$i_n = 2 i_{n-2} + 2(n-2) i_{n-3} + (n-2)(n-3)i_{n-4}.$$
\end{lem}
\begin{proof}
Since $i_n = i_{n-1} + (n-1) i_{n-2}$, we have
\begin{align*}
i_n &= i_{n-1} + (n-1) i_{n-2} \\
    &= (i_{n-2} + (n-2)i_{n-3}) + (n-1) i_{n-2} \\
    &= 2 i_{n-2} + (n-2)(i_{n-2} + i_{n-3}) \\
    &= 2 i_{n-2} + (n-2)(2 i_{n-3} + (n-3)i_{n-4} ) \\
    &= 2 i_{n-2} + 2(n-2) i_{n-3} + (n-2)(n-3)i_{n-4}.
\end{align*}
\end{proof}

\begin{prop}\label{kis2} Let $H=\k^G \# \k F$ be the bismash product arising from an exact factorization of $S_n = S_{n-2}\cdot G$.  Then the Schur indicator of any irreducible $H$-module $V$ induced from $x\in G$ satisfies:
$$\nu^{[2]}(V) = \begin{cases} 1 & \mbox{ if and only if}\; \mathcal{O}_x\; \mbox{is not a null indicator orbit.} \\
                               0 & \mbox{ if and only if}\; \mathcal{O}_x\; \mbox{is a null indicator orbit.} \end{cases}$$
\end{prop}
\begin{proof}
Recall that $Tr(S) = \sum_{\chi} \nu^{[2]}(\chi)\chi(1)= i_n$ for the Hopf algebra $H$, and $i_n = \sum_{i=1}^{c_{n-2}} d_{n,i}$ since $S_n$ is totally orthogonal.  Now by our description of $H$-modules which may non-zero Schur indicators in \Cref{moddesck2}, by \Cref{involk2} we must have $\nu^{[2]}(\chi)=1$ for all irreducible modules induced from $F$-orbits that are not null indicator orbits. \end{proof}

When we have an exact factorization $S_n=FG$ with $k=3$ and so $F=S_{n-3}$ fixing $n-2,n-1,n$ and $G$ is sharply 3-transitive on $\Omega$.  By \Cref{Gident}, each $g\in G$ is identified with an ordered triple $g\simeq [i,j,l]$ where $(n-2)^g=i,\; (n-1)^g=j,\; n^g = l$.  We use the same notation of \Cref{modulenumbers} for the description of simple $S_n$ modules for any $n$. 

\begin{lem} For $g \simeq [i,j,k]$ via \Cref{Gident}, $$F_g = \{ \sigma \in F : i^\sigma = i, j^\sigma = j, k^\sigma = k \}.$$
\end{lem}
\begin{proof}
\Cref{stabdesc} with $k=3$. \end{proof}

\begin{lem} The number of $F$-orbits in $G$ is 34. \end{lem}
\begin{proof}
Since $g\ltri f = g$ if and only if $f$ fixes $i,j,k$ for $g\simeq [i,j,k]$, we see that if $f$ has $m$ fixed points as a permutation in $S_{n-3}$, then $f$ fixes exactly $(m+3)(m+2)(m+1)$ many tuples $[i,j,k]\simeq g$.  Indeed for if $f$ has $m$ fixed points in $S_{n-3}$, then it has $m+3$ fixed points as a permutation in $S_{n}$.  If we let $N$ be the number of $F$-orbits and $\phi(m)$ be the number of $f\in F$ with $m$ fixed points, we see that

\begin{align*}
N  &= \frac{1}{(n-3)!} \sum_{\sigma\in F} \fix(\sigma) \\
   &= \frac{1}{(n-3)!} \sum_{m=0}^{n-3} \phi(m) (m+1)(m+2)(m+3) \\
   &= \frac{1}{(n-3)!} \sum_{m=0}^{n-3}  m^3 \phi(m) + 6m^2\phi(m) +11m\phi(m) + 6\phi(m)  \\
   &= \frac{1}{(n-3)!} \left[ 5(n-3)! + 12(n-3)! + 11(n-3)! + 6(n-3)! \right] \\
   &= 34.
\end{align*}
\end{proof}

\begin{lem} There are exactly 20 null indicator orbits. \end{lem}
\begin{proof}
By \Cref{NIOdesc} and \Cref{stabfamcor}:
The $0$-family has 6 orbits, 2 of which are null indicator orbits.  The $1$-family has 18 orbits, 12 of which are null indicator orbits. The $2$-family has 9 orbits, 6 of which are null indicator orbits. The $3$-family has 1 orbit and this orbit is not a null indicator orbit.
\end{proof}

\begin{prop}\label{moddesck3} Enumerate the simple $S_n$ modules with $\{1, \ldots, c_n\}$ with respective dimensions $d_{n, i}$.  The total number of simple $H=\k ^G \# \k S_{n-2}$ modules which may have non-zero indicator values is $4c_{n-3} + 6c_{n-4} + 3c_{n-5} + c_{n-6}$.  The dimension of each module is contingent upon the $F$-orbits of $G$ as follows:  For $V$ the $i^{\mbox{\tiny{th}}}$ irreducible $F_x$ module, then for the induced $H$-module $\widehat{V}$:
\begin{center}
\begin{tabular}{r l}
 $\dimn (\widehat{V}) = d_{n-3,i}$ & if $F_x = S_{n-3}$ \\
 $\dimn (\widehat{V}) = (n-3)d_{n-4,i}$ & if $F_x = S_{n-4}$ \\
 $\dimn (\widehat{V}) = (n-3)(n-4)d_{n-5,i}$ & if $F_x = S_{n-5}$  \\
 $\dimn (\widehat{V}) = (n-3)(n-4)(n-5)d_{n-6,i}$ & if $F_x = S_{n-6}$.
\end{tabular}
\end{center}
\end{prop}
\begin{proof}
Since $\k F$ is a free $\k F_x$-module of rank $[F : F_x]$, the statement about dimensions is clear.  Any irreducible $H$-module is induced from a $F_x$-module for a given set of $F$-orbit representatives.  We have shown that 2 of the orbits with stabilizer $S_{n-3}$, 12 orbits with stabilizer $S_{n-4}$ and 6 orbits with stabilizer $S_{n-5}$ are null indicator orbits.  Thus the remaining orbits may have non-zero Schur incicators and the proof is complete. \end{proof}

\begin{lem}\label{involk3}
Let $i_n$ denote the number of involtions of $S_n$, then $$i_n = 4i_{n-3} + 6(n-3)i_{n-4} + 3(n-3)(n-4)i_{n-5} + (n-3)(n-4)(n-5)i_{n-6}.$$
\end{lem}
\begin{proof}
Since $i_n = i_{n-1} + (n-1) i_{n-2}$, we have
\begin{align*}
i_n &= i_{n-1} + (n-1) i_{n-2} \\
    &= n i_{n-2} + (n-2) i_{n-3} \\
    &= 2(n-2)i_{n-3} + n(n-3)i_{n-4} \\
    &= 4i_{n-3} + 2(n-3)i_{n-3} + n(n-3)i_{n-4} \\
    &= 4i_{n-3} + (n-3)(n+2)i_{n-4} + 2(n-3)(n-4)i_{n-5} \\
    &= 4i_{n-3} + 6(n-3)i_{n-4} + (n-3)(n-4)i_{n-4} + 2(n-3)(n-4)i_{n-5} \\
    &= 4i_{n-3} + 6(n-3)i_{n-4} + 3(n-3)(n-4)i_{n-5} + (n-3)(n-4)(n-5) i_{n-6}. \end{align*}
\end{proof}

\begin{prop}\label{kis3} Let $H=\k^G \# \k F$ be the bismash product arising from an exact factorization of $S_n = S_{n-3}\cdot G$.  Then the Schur indicator of any irreducible $H$-module $V$ induced from $x\in G$ satisfies:
$$\nu^{[2]}(V) = \begin{cases} 1 & \mbox{ if and only if}\; \mathcal{O}_x\; \mbox{is not a null indicator orbit.} \\
                               0 & \mbox{ if and only if}\; \mathcal{O}_x\; \mbox{is a null indicator orbit.} \end{cases}$$
\end{prop}
\begin{proof}
Recall that $Tr(S) = \sum_{\chi} \nu^{[2]}(\chi)\chi(1)= i_n$ for the Hopf algebra $H$, and $i_n = \sum_{i=1}^{c_{n-2}} d_{n,i}$ since $S_n$ is totally orthogonal.  Now by our description of $H$-modules which may non-zero Schur indicators in \Cref{moddesck3}, by \Cref{involk3} we must have $\nu^{[2]}(\chi)=1$ for all irreducible modules induced from $F$-orbits that are not null indicator orbits. \end{proof}

By now, the method is clear. For an exact factorization $S_n=FG$ with $k=4$ and so $F=S_{n-4}$ fixing $n-3, n-2, n-1, n$ and $G$ is sharply 4-transitive on $\Omega$.  By \Cref{Jo}, this is only possible when $n=11$ and $G=M_{11}$  By \Cref{Gident}, each $g\in G$ is identified with an ordered triple $g\simeq [a,b,c,d]$ where $8^g=a,\; 9^g=b,\; 10^g = c,\; 11^g=d$.  We use the same notation of \Cref{modulenumbers} for the description of simple $S_n$ modules for any $n$.  In this section, we skip proofs which follow \textit{mutatis mutandis} from the previous sections. 

\begin{lem} For $g \simeq [a,b,c,d]$ via \Cref{Gident}, we have $$F_g = \{ \sigma \in F : a^\sigma = a, b^\sigma = b, c^\sigma = c , d^\sigma = d\}.$$
\end{lem}
\begin{proof}
\Cref{stabdesc} with $k=4$. \end{proof}

\begin{lem} The number of $F$-orbits in $G$ is 209. \end{lem}
\begin{proof}
Since $g\ltri f = g$ if and only if $f$ fixes $a,b,c,d$ for $g\simeq [a,b,c,d]$, we see that if $f$ has $m$ fixed points as a permutation in $S_{7}$, then $f$ fixes exactly $(m+4)(m+3)(m+2)(m+1)$ many tuples $[a,b,c,d]\simeq g$.  Indeed for if $f$ has $m$ fixed points in $S_{7}$, then it has $m+4$ fixed points as a permutation in $S_{11}$.  If we let $N$ be the number of $F$-orbits and $\phi(m)$ be the number of $f\in S_7$ with $m$ fixed points, we see that

\begin{align*}
N  &= \frac{1}{7!} \sum_{\sigma\in S_7} \fix(\sigma) \\
   &= \frac{1}{7!} \sum_{m=0}^{7} \phi(m) (m+1)(m+2)(m+3)(m+4) \\
   &= 209.
\end{align*}
where the last equality follows direct computation.
\end{proof}

\begin{lem}There are 166 null indicator orbits and 43 orbits that are not. \end{lem}
\begin{proof}
By \Cref{NIOdesc} and \Cref{stabfamcor}:
The $0$-family has 24 orbits, 14 of which are null indicator orbits. The $1$-family has 96 orbits, 80 of which are null indicator orbits. The $2$-family has 72 orbits, 60 of which are null indicator orbits. The $3$-family has 16 orbits, 12 of which are null indicator orbits. The $4$-family has 1 orbit and this orbit is not a null indicator orbit.
\end{proof}

\begin{lem}\label{moddesck4} Enumerate the simple $S_n$ modules with $\{1, \ldots, c_n\}$ with respective dimensions $d_{n, i}$.  The total number of simple $H=\k ^G \# \k S_{n-2}$ modules which may have non-zero indicator values is $10c_7 + 16c_6 + 12c_5 + 4c_4 + c_3$.  The dimension of each module is contingent upon the $F$-orbits of $G$ as follows:  For $V$ the $i^{\mbox{\tiny{th}}}$ irreducible $F_x$ module, then for the induced $H$-module $\widehat{V}$:
\begin{center}
\begin{tabular}{r l}
 $\dimn (\widehat{V}) = d_{7,i}$ & if $F_x = S_7$ \\
 $\dimn (\widehat{V}) = 7 d_{6,i}$ & if $F_x = S_{6}$ \\
 $\dimn (\widehat{V}) = 42 d_{5,i}$ & if $F_x = S_{5}$  \\
 $\dimn (\widehat{V}) = 210 d_{4,i}$ & if $F_x = S_{4}$ \\
 $\dimn (\widehat{V}) = 840 d_{3,i}$ & if $F_x = S_{3}$.
\end{tabular}
\end{center}
\end{lem}

\begin{lem} $i_{11} = 10\cdot i_7 + 16\cdot(7\cdot i_6) + 12\cdot(42\cdot i_5) + 4\cdot(210\cdot i_4) + 840\cdot i_3$ \end{lem}

\begin{prop}\label{kis4} Let $H=\k^G \# \k F$ be the bismash product arising from the exact factorization of $S_{11} = S_7\cdot M_{11}$.  Then the Schur indicator of any irreducible $H$-module $V$ induced from $x\in G$ satisfies:
$$\nu^{[2]}(V) = \begin{cases} 1 & \mbox{ if and only if}\; \mathcal{O}_x\; \mbox{is not a null indicator orbit.} \\
                               0 & \mbox{ if and only if}\; \mathcal{O}_x\; \mbox{is a null indicator orbit.} \end{cases}$$
\end{prop}

For an exact factorization $S_n=FG$ with $k=5$ and so $F=S_{n-5}$ fixing $n-4, n-3, n-2, n-1, n$ and $G$ is sharply 5-transitive on $\Omega$.  By \Cref{Jo}, this is only possible when $n=12$ and $G=M_{12}$  By \Cref{Gident}, each $g\in G$ is identified with an ordered triple $g\simeq [a,b,c,d, e]$ where $8^g=a,\; 9^g=b,\; 10^g = c,\; 11^g=d,\; 12^g=e$.  We use the same notation of \Cref{modulenumbers} for the description of simple $S_n$ modules for any $n$.  In this section, we skip proofs which follow \textit{mutatis mutandis} from the previous sections. 

\begin{lem} For $g \simeq [a,b,c,d,e]$ via \Cref{Gident}, we have $$F_g = \{ \sigma \in F : a^\sigma = a, b^\sigma = b, c^\sigma = c , d^\sigma = d, e^\sigma=e\}.$$
\end{lem}
\begin{proof}
\Cref{stabdesc} with $k=5$. \end{proof}

\begin{lem} The number of $F$-orbits in $G$ is 1546. \end{lem}
\begin{lem}There are exactly 1404 null inidcator orbits. \end{lem}
\begin{proof}
By \Cref{NIOdesc} and \Cref{stabfamcor}:
The $0$-family has 120 orbits, 94 of which are null indicator orbits. The $1$-family has 600 orbits, 550 of which are null indicator orbits. The $2$-family has 600 orbits, 560 of which are null indicator orbits. The $3$-family has 200 orbits, 180 of which are null indicator orbits. The $4$-family has 25 orbits, 20 of which are null indicator orbits. The $5$-family has 1 orbit and this orbit is not a null indicator orbit.
\end{proof}

\begin{prop}\label{moddesck5} Enumerate the simple $S_n$ modules with $\{1, \ldots, c_n\}$ with respective dimensions $d_{n, i}$.  The total number of simple $H=\k ^G \# \k S_{n-2}$ modules which may have non-zero indicator values is $26c_{7} + 50c_{6} + 40c_{5} + 20c_{4} + 5c_3 + c_2$.  The dimension of each module is contingent upon the $F$-orbits of $G$ as follows:  For $V$ the $i^{\mbox{\tiny{th}}}$ irreducible $F_x$ module, then for the induced $H$-module $\widehat{V}$:
\begin{center}
\begin{tabular}{r l}
 $\dimn (\widehat{V}) = d_{7,i}$ & if $F_x = S_7$ \\
 $\dimn (\widehat{V}) = 7 d_{6,i}$ & if $F_x = S_{6}$ \\
 $\dimn (\widehat{V}) = 42 d_{5,i}$ & if $F_x = S_{5}$  \\
 $\dimn (\widehat{V}) = 210 d_{4,i}$ & if $F_x = S_{4}$ \\
 $\dimn (\widehat{V}) = 840 d_{3,i}$ & if $F_x = S_{3}$ \\
 $\dimn (\widehat{V}) = 2520 d_{2,i}$ & if $F_x = S_{2}$.
\end{tabular}
\end{center}
\end{prop}

\begin{lem} $i_{12} = 26\cdot i_7 + 50\cdot(7\cdot i_6) + 40\cdot(4\cdot2 i_5) + 20\cdot(210\cdot i_4) + 5\cdot(840\cdot i_3) + 2520\cdot i_2$ \end{lem}

\begin{prop}\label{kis5} Let $H=\k^G \# \k F$ be the bismash product arising from the exact factorization of $S_{12} = S_7\cdot M_{12}$.  Then the Schur indicator of any irreducible $H$-module $V$ induced from $x\in G$ satisfies:
$$\nu^{[2]}(V) = \begin{cases} 1 & \mbox{ if and only if}\; \mathcal{O}_x\; \mbox{is not a null indicator orbit.} \\
                               0 & \mbox{ if and only if}\; \mathcal{O}_x\; \mbox{is a null indicator orbit.} \end{cases}$$
\end{prop}

\section{Main Results and Further Questions}

We summarize the last section in the following theorem.

\begin{thm}\label{TGfact} Let $S_n = S_{n-k} \cdot G$ be an exact factorization of $S_n$.  For the associated bismash product $H=\k^G\# \k S_{n-k}$, every irreducible $H$-module has Schur indicator 1 or 0.  Moreover, any $H$-module with indicator 0 is induced from a null indicator orbit. \end{thm}

The following example shows we cannot extend these results with $F=\widetilde{S}_{n-k}$.

\begin{ex}\label{cntrex} \rm{Let $S_7 = \WTS{5}\cdot G$ with  $F=\lan (1,2,3,4,5), (1,2)(6,7)\ran$ and $G= \lan (1,2,3,4,5,6,7), (1,3,2,6,4,5)\ran\cong AGL(1,7)$.  Let $H=\k^G \# \k F$ be the bismash product induced from exact factorization.  By use of \cite{GAP}, for $x=(1,2,6)(4,7,5)$, we have $F_x = \lan (2,4)(6,7), (1,5)(2,3,4)(6,7) \ran \cong D_{12}$ and for the following character table of $F_x$:
\begin{center}
\begin{tabular}{r | c| c| c| c| c| c|}
          & $(\, )$ & $(3,4)(6,7)$ & $(2,3,4)$ & $(1,5)(6,7)$ & $(1,5)(3,4)$ & $(1,5)(2,3,4)(6,7)$ \\
 \hline
 $\chi_1$ &   $1$   &      $1$     &    $1$    &    $1$       &      $1$     & $1$      \\
 \hline
 $\chi_2$ &   $1$   &      $-1$    &    $1$    &    $-1$      &      $1$     & $-1$      \\
 \hline
 $\chi_3$ &   $1$   &      $-1$    &    $1$    &    $1$       &      $-1$    & $1$       \\
 \hline
 $\chi_4$ &   $1$   &      $1$     &    $1$    &    $-1$      &      $-1$    & $-1$      \\
 \hline
 $\chi_5$ &   $2$   &      $0$     &    $-1$   &    $-2$      &      $0$     & $1$        \\
 \hline
 $\chi_6$ &   $2$   &      $0$     &    $-1$   &    $2$       &      $0$     & $-1$       \\
 \end{tabular}
 \end{center}
 The irreducible $H$-module induced from $\chi_5$ has Schur indicator value $-1$. }
 \end{ex}
 
This example also illustrates the point that if $F\cong F'$ and $G\cong G'$, the Hopf algebras $H= \k^G \# \k F$ and $H' = \k^{G'} \# \k F'$ need not be isomorphic, even if $FG=L=F'G'$ are exact factorizations of the same group.  Moreover, this example is not isolated and occurs infinitely often, with suitable conditions for $n$.
 
\begin{conj}\label{WTSSn} Let $S_n=FG$ be an exact factorization of $S_n$ with $F\cong \widetilde{S}_{n-k}$ as in \Cref{Fknown} for $k=2$ or $3$ and $G$ sharply $k$-transitive.  Then $H= \k^G \# \k F$ has an irreducible module with Schur indicator -1. \end{conj}

We now address the cases when $F$ is an alternating group.
\begin{thm}\label{ANdual} For $n\geq 4$, the Schur indicator, $\nu^{[2]}(\widehat{\chi})$, of an irreducible character of $H=\k^{A_n}\# \k C_2$ is 0 or +1. \end{thm}
\begin{proof}
Let $\tau\in S_n$ be an odd transposition.  Since $A_n\ltri S_n$, the action of $A_n$ on $F=\lan \tau\ran$ is trivial and the action of $\tau$ on $G=A_n$ is by conjugation.  This simplifies equation \Cref{FSL} to 
\num
\begin{equation}\label{eq:3}
\Lambda^{[2]}=\frac{1}{2}\sum_{y\in A_n} \delta_{y\inv,y} p_y\#1 + \delta_{y\inv, y^\tau} p_y\#1
\end{equation}
We proceed in two cases, if $F_x=\{()\}$ and if $F_x=C_2$ for $x\in A_n$.  Suppose $F_x=\{()\}$ and thus $x^\tau\neq x$ and our coset representatives are $\{(),\tau\}$.  This gives the Frobenius-Schur indicator for $\widehat{\chi}$ is 
\begin{align*}
\nu^{[2]}(\widehat{\chi})&=\frac{1}{2}\sum_{y\in A_n} \delta_{y\inv , y}(\delta_{y,x}\chi(1)+\delta_{y^\tau, x}\chi(1))+ \delta_{y\inv ,y^\tau}(\delta_{y,x\chi(1)}+\delta_{y^\tau, x}\chi(1)) \\
&= \frac{1}{2}\sum_{y\in A_n} (\delta_{y\inv, y} + \delta_{y\inv , y^\tau} ) (\delta_{y,x}+\delta_{y^\tau, x} )\chi(1) \\
&=\frac{1}{2}\left[(\delta_{x\inv , x}+\delta_{x\inv , x^\tau})+(\delta_{(x^\tau)\inv , x^{\tau}}+ \delta_{(x^\tau)\inv ,x})\right].
\end{align*}
Indeed for if $F_x$ is trivial, $\chi(1)=1$ and we have non-trivial contributions to the sum only if $y=x$ or $y=x^\tau$.

Suppose also that $x=x\inv$, then $\nu^{[2]}(\widehat{\chi})=1$. If instead, $x\neq x\inv$, the FS indicator is +1 if and only if $x^{-1}=x^\tau$.  Else, the $\nu^{[2]}(\widehat{\chi})=0$.  This completes the first case.

Now suppose $F_x=F$ and we take $\{( )\}$ for the set of coset representatives.  Now the $\nu^{[2]}(\widehat{\chi})$ is given by
\begin{align*}
\nu^{[2]}(\widehat{\chi}) &= \frac{1}{2} \sum_{y\in A_n} \delta_{y\inv , y}\left(\delta_{y,x} \chi(1)\right)+\delta_{y\inv , y^\tau}\left(\delta_{y,x} \chi(1)\right) \\
&= \frac{1}{2}\left(\delta_{x\inv, x}+\delta_{x\inv, x^\tau}\right) \\
&= \delta_{x\inv,x}.
\end{align*}
Thus the $\nu^{[2]}(\widehat{\chi})$ is 1 if and only if $x=x\inv$, else it is 0. \end{proof}

By considering the dual Hopf algebra to that of \Cref{ANdual}, the same conclusion also seems to hold.  The result will depend upon the understanding the relation between representations of $A_n$ and $S_n$.

\begin{conj}\label{ANF} For $n\geq 3$, the Schur indicator, $\nu^{[2]}(\widehat{\chi})$, of an irreducible character of $H=\k^{C_2}\# \k A_n$ is 0 or 1. \end{conj}
\begin{proof}
Let $C_2 = \lan x \ran$ and $F:=A_n$.  Since $F_{()} = F= A_n$, we must have that $F_x = F = A_n$.  Thus the Schur indicator values of $H$ are the usual Frobenius-Schur indicator values of the group $A_n$ when we lift from the identity.  By \Cref{Bob}, these values are 1 or 0. 

When we lift from $\tau$, we have something else to compute.  Since $x\ltri a=x$ for all $a\in A_n$ , $x\rtri a=a^x$.  This gives us 
$$\nu^{[2]} \left(\widehat{\chi}\right) = \frac{2}{n!} \sum_{a\in A_n} \chi(a^x a).$$

The result should now follow from the known theory of twisted Frobenius-Schur indicators.
\end{proof}

We believe these conjectures to hold, since \Cref{ANF} has been checked up to $n=13$ and there is a discernable pattern that is explained by the theory of twisted Frobenius-Schur indicators.  Moreover, for $n\leq 10$ all factorizations have been checked by use of \cite{GAP} and we have:

\begin{thm} Let $S_n = FG$ be an exact factorization for $n\leq 10$.  Let $p$ be the largest prime $p<n-2$ and assume $F$ contains the p-cycle $(1,2,\ldots, p)$.  Let $\Gamma$ be the $F$-orbit of $\{1,2,\ldots , p\}$ and let $H=\k ^G \# \k F$ be the induced bismash product Hopf algebra.  If $F\neq \widetilde{S}_\Gamma$ as described in \Cref{Fknown}, then the Schur indicator for any irreducible character of $H$ is 1 or 0.
\end{thm}
\begin{proof} Most cases are handled by \Cref{TGfact}.  The remaining cases were checked by \cite{GAP}.

\end{proof}
This has lead us to make the following conjecture, which would make the situation in \Cref{WTSSn} the only ``badly behaved" factorizations. 
\begin{conj}\label{conjecture} Let $S_n = F\cdot G$ be an exact factorization with $|F|> |G|$ and $F\neq \widetilde{S}_{n-k}$ as in \Cref{Fknown}.  Then the Schur indicator of an irreducible $H=\k^G \# \k F$ module is $0$ or $1$. \end{conj}
Moreover, there are patterns of stabilizers and $F$-orbits, similar to \Cref{TGfact} that lead me to believe a general proof is not out of reach with these techniques.

A quick word about characteristic $p$.  If $\k$ has characteristic $p\neq 2$, \cite{JM2} gives the following result:
\begin{thm}\cite[Corollary 6.6]{JM2}
For a matched pair of groups $(F,G,\rtri, \ltri)$, let $H_{\BC} = \BC^G \# \BC F$ be the bismash product defined over $\BC$ and $H_{\k} = \k^G \# \k F$ the bismash product defined over $\k$, algebraically closed of characteristic $p\neq 2$.  Then \begin{enumerate}
\item If all irreducible $H_\BC$-modules have indicator $+1$, then the same is true for all irreducible $H_\k$-modules.
\item If all irreducible $H_\BC$-modules have indicator $0$ or $1$, then the same is true for all irreducible $H_\k$-modules.
\item If all irreducible $H_\BC$-modules are self dual, then the same is true for all irreducible $H_\k$-modules.
\end{enumerate}
\end{thm}
Thus, our main result will hold in characteristic $p>0$. Similarly the conjectures will also hold, provided of course they are true.

\section{Acknowledgements}

This work is formed the bulk of the author's Ph.D. thesis for the University of Southern California.  He would first and foremost like to thank his thesis advisor, Dr. Susan Montgomery.  He cannot thank her enough for her support, numerous discussions and pointed advice.  Without her, this work simply would not exist.  Also, the author would like to thank Bob Guralnick for his useful comments, ideas and excellent sources on group theory.


\end{document}